\documentclass[a4,12pt]{amsart}
\usepackage{amssymb,amstext,amsmath,amscd,amsthm,amsfonts,enumerate,latexsym}
\usepackage{color}
\usepackage{graphicx}
\usepackage[all]{xy}
\usepackage{stmaryrd} 

\usepackage{hyperref}
\usepackage{bm}

\theoremstyle{plain}
\newtheorem{thm}{Theorem}[section]

\newtheorem*{thm*}{Theorem}
\newtheorem*{cor*}{Corollary}

\newtheorem{lem}[thm]{Lemma}
\newtheorem{cor}[thm]{Corollary}

\newtheorem*{claim*}{Claim}

\theoremstyle{definition}

\newtheorem{ex}[thm]{Example}

\newtheorem{setup}[thm]{Setup}

\theoremstyle{remark}
\newtheorem{rem}[thm]{Remark}

\numberwithin{equation}{thm}

\newtheorem*{ac}{Acknowledgments}

\def\Im{\operatorname{Im}}

\def\m{\mathfrak m}
\def\n{\mathfrak n}

\newcommand{\rma}{\mathrm{a}}

\newcommand{\rme}{\mathrm{e}}

\newcommand{\rmr}{\mathrm{r}}

\newcommand{\rmI}{\mathrm{I}}

\newcommand{\rmK}{\mathrm{K}}

\newcommand{\calF}{\mathcal{F}}
\newcommand{\calG}{\mathcal{G}}

\newcommand{\mapright}[1]{%
\smash{\mathop{%
\hbox to 1cm{\rightarrowfill}}\limits^{#1}}}

\newcommand{\mapleft}[1]{%
\smash{\mathop{%
\hbox to 1cm{\leftarrowfill}}\limits_{#1}}}

\def\gr{\mbox{\rm gr}}

\def\ord{\operatorname{ord}}

\title[Almost Gorenstein standard graded rings and their localizations]{Almost Gorenstein standard graded rings \\ and their localizations}

\author[Naoki Endo]{Naoki Endo}
\address{School of Political Science and Economics, Meiji University, 1-9-1 Eifuku, Suginami-ku, Tokyo 168-8555, Japan}
\email{endo@meiji.ac.jp}
\urladdr{https://www.isc.meiji.ac.jp/~endo/}

\thanks{2020 {\em Mathematics Subject Classification.} 13H10, 13A02, 13B35}
\thanks{{\em Key words and phrases.} almost Gorenstein ring, standard graded ring, associated graded module}

\thanks{The author was partially supported by JSPS Grant-in-Aid for Scientific Research (C) 23K03058 and 26K06773.}

\begin{document}

\maketitle

\setlength{\baselineskip} {15.3pt}

\begin{abstract}
We prove that a Cohen-Macaulay standard graded integral domain over a field is almost Gorenstein as a graded ring if and only if the completion of its localization at the graded maximal ideal is almost Gorenstein as a local ring. 
\end{abstract}



\section{Introduction}

The aim of this paper is to prove the following.

\begin{thm}\label{main}
Let $A=\bigoplus_{i \ge 0}A_i$ be a Cohen-Macaulay standard graded integral domain over a field $k$ and $\n$ the graded maximal ideal of $A$. Then the $\n A_{\n}$-adic completion $\widehat{A_{\n}}$ is almost Gorenstein as a local ring if and only if $A$ is almost Gorenstein as a graded ring.
\end{thm}

The notion of an almost Gorenstein ring originated in the work of Barucci and Fr\"oberg \cite{BF} in the one-dimensional analytically unramified case. It was subsequently reformulated for one-dimensional Cohen-Macaulay local rings by Goto, Matsuoka, and Phuong \cite{GMP}, and was extended to Cohen-Macaulay local and graded rings of arbitrary dimension by Goto, Takahashi, and the author \cite{GTT}, using Ulrich modules. 
A fundamental issue in the theory is the relationship between the graded and local almost Gorenstein properties. If a Cohen-Macaulay graded ring $A=\bigoplus_{i \ge 0}A_i$ is almost Gorenstein as a graded ring, then its localization $A_{\n}$ at the graded maximal ideal $\n$ is almost Gorenstein as a local ring. The converse, however, fails in general. This naturally raises the question of when the almost Gorenstein property can be recovered from the localization, or from its completion.

Theorem \ref{main} provides a complete answer to this question for Cohen-Macaulay standard graded integral domains over a field. The assumption that the grading is standard cannot be omitted, as shown by \cite[Example 8.8]{GTT}. Moreover, the corresponding assertion need not hold for standard graded algebras over a local ring; see \cite[Theorems 2.7, 2.8]{GMTY1}. Nor can the assumption that $A$ is an integral domain be dropped: the conclusion need not hold even when $A$ is reduced; see \cite[Example 3.4]{EM} and \cite[Theorem 3.9]{GTT}.

The main difficulty in the converse implication is that localization and completion no longer retain the grading, whereas the graded almost Gorenstein property is formulated in terms of a homogeneous embedding into the graded canonical module. To overcome it, we introduce a filtration on the canonical module of the completion $\widehat{A_{\n}}$ induced by the grading on the canonical module of $A$. The associated graded module with respect to this filtration recovers a suitable shift of the graded canonical module of $A$. This enables us to relate an arbitrary embedding of the completion into its canonical module to the homogeneous initial form of the image of $1$, and thereby to derive an explicit formula for the multiplicity of the cokernel.


\section{Proof of Theorem \ref{main}}

We begin by recalling the notation for multiplicities that will be essential to our argument.

\subsection{Hilbert-Samuel multiplicity}
Let $(R,\m)$ be a Noetherian local ring and $N$ a non-zero finitely
generated $R$-module. Let $\calF=\{F_i\}_{i\geq 0}$ be a {\it good
$\m$-filtration} of $N$, that is, $\calF$ is a descending chain 
$
N=F_0 \supseteq F_1 \supseteq \cdots  \supseteq F_i \supseteq \cdots
$
of $R$-submodules of $N$ such that $\m F_i\subseteq F_{i+1}$ for all $i\geq 0$ and $\m F_i=F_{i+1}$ for $i\gg 0$.
The function $i \mapsto \ell_R(N/F_{i+1})$
agrees, for $i\gg 0$, with a polynomial of degree $\dim_R N$, where $\ell_R(-)$ denotes the length as an $R$-module.
We denote its normalized leading coefficient by $\rme_{\calF}(N)$, 
and call it the {\it Hilbert-Samuel multiplicity} of $N$ with respect to
$\calF$. In particular, for the $\m$-adic filtration
$\{\m^iN\}_{i\geq 0}$, we write $\rme_{\m}^0(N)$ for the corresponding Hilbert-Samuel multiplicity. 
The multiplicity is independent of the choice of a good $\m$-filtration; indeed, $\rme_{\calF}(N)=\rme_{\m}^{0}(N)$ for every such filtration $\calF$. See \cite[Section 1.3]{Rossi-Valla} for details. 

\subsection{Graded multiplicity}
Let $A=\bigoplus_{i\ge 0}A_i$ be a standard graded algebra
over a field $k=A_0$, and let $L$ be a non-zero finitely generated
graded $A$-module. We denote by
$$
H_L(t)=\sum_{i\in\Bbb Z}(\dim_kL_i)t^i
$$
the Hilbert series of $L$. Set $d=\dim_A L$, and write
$$
H_L(t)=\frac{h_L(t)}{(1-t)^d}
$$
where $h_L(t)$ is a Laurent polynomial with $h_L(1)\neq 0$.
The {\it multiplicity} of $L$ is defined by 
$$
\rme(L)=h_L(1).
$$ 
Equivalently, if $d>0$, $\rme(L)$ is the normalized leading coefficient of the Hilbert polynomial of $L$. Writing $\n=A_+$, we have $\rme(L)=\rme_{\n A_{\n}}^0(L_{\n})$
, where $A_+ = \bigoplus_{i > 0}A_i$.  
We also recall the relation with the Hilbert-Samuel multiplicity of a good filtration. Let $(R,\m)$ be a Noetherian local ring, $N$ a non-zero finitely generated $R$-module, and $\calF=\{F_i\}_{i\geq 0}$ a good $\m$-filtration of $N$. Then
$\rme_{\calF}(N)=\rme(\gr_{\calF}(N))$, 
where $\gr_{\calF}(N) = \bigoplus_{i \ge 0}F_i/F_{i+1}$ is regarded as a graded module over the standard graded algebra $\gr_{\m}(R)=\bigoplus_{i \ge 0}\m^i/\m^{i+1}$ over $R/\m$. 


\medskip

Throughout this paper, we adopt the following setting.

\begin{setup}\label{setup}
Let $A=\bigoplus_{i \ge 0}A_i$ be a Cohen-Macaulay standard graded integral domain over a field $k=A_0$ and $\n$ the graded maximal ideal of $A$. We denote by $R=\widehat{A_{\n}}$ the $\n A_{\n}$-adic completion of the local ring $A_{\n}$ and set $\m=\n R$. Let $\rmK_A$ be the graded canonical module of $A$ and $a=\rma(A)$ the $a$-invariant of $A$. Set $M=\rmK_A(-a)$.
\end{setup}

We first write the Hilbert series of $A$ as
\[
H_A(t)=\frac{h_A(t)}{(1-t)^d} \ \ \text{and} \ \ h_A(t)=h_0+h_1t+\cdots+h_st^s
\]
where $d=\dim A$ and $h_s\ne0$. Since $A$ is Cohen-Macaulay, it follows that $h_j\geq0$ for all $0\leq j\leq s$, and $a=s-d$. By the proof of \cite[4.4 Theorem]{Stanley} (see also \cite[Corollary 4.4.6]{BH}), we obtain
$$
H_M(t) = t^a H_{\rmK_A}(t) = \frac{t^{a+d}h_A(t^{-1})}{(1-t)^d} =  \frac{h_s + h_{s-1}t+ \cdots + h_0t^s}{(1-t)^d}.
$$
Thus, $M_i=(0)$ for all $i<0$ and $\dim_kM_0=h_s>0$. Hence
$M=\bigoplus_{i\geq0}M_i$ and $M_0\ne(0)$.

Choose homogeneous elements $x_1,x_2,\ldots,x_r$ of $M$ such that $M=\sum_{j=1}^r Ax_j$, where $r=\rmr(A)=\mu_A(M)$ is the Cohen-Macaulay type of $A$. Here, $\mu_A(-)$ denotes the minimal number of generators. Set $d_j=\deg x_j$ for each $1\leq j\leq r$. Since $M_i=(0)$ for all $i<0$, we have $d_j \ge 0$. Moreover, since $A$ is standard graded, we see that $\n^{\ell} = \bigoplus_{j \ge \ell}A_j$ for every $\ell\ge 0$. For each $i \ge 0$, set $E_i=\bigoplus_{j\geq i}M_j$. Then
$$
E_i=\sum_{j=1}^r\n^{q_j}x_j
$$
where $q_j=\max\{i-d_j,0\}$. Hence, setting $b=\max\{d_j\mid 1\leq j\leq r\}$, we have
$$
\n^i M \subseteq E_i \subseteq \n^{\max\{i-b, 0\}}M
$$
for all $i \ge 0$. In particular, for every $i\geq b$, we obtain $\n^i M \subseteq E_i \subseteq \n^{i-b}M$. 
It follows that the filtration $\{E_i\}_{i\geq0}$ and the
$\n$-adic filtration $\{\n^i M\}_{i\geq0}$ define the same topology
on $M$.
Consequently, setting $K=M \otimes_AR$, we have $K\cong\rmK_R$ and
$$
K \cong \widehat{M} = \varprojlim_{i}M/\n^i M \cong \varprojlim_{i}M/E_i \cong \varprojlim_{i} \left(\bigoplus_{0 \le j < i}M_j\right) \cong \prod_{j \ge 0}M_j
$$
where $\widehat{M}$ denotes the $\n$-adic completion of $M$.

For each $i \ge 0$, since $R$ is flat over $A$, the exact sequence $0 \to  E_i \to  M \to M/E_i \to 0$ induces an exact sequence
$$
0 \to E_i \otimes_A R \to M \otimes_A R \to (M/E_i) \otimes_AR \to 0
$$
of $R$-modules. As $M/E_i$ is naturally an $A/\n^i$-module, we have
\begin{eqnarray*}
(M/E_i) \otimes_AR \!\!&\cong&\!\! (M/E_i) \otimes_{A/\n^i}(R/\n^i R) \cong (M/E_i) \otimes_{A/\n^i}(A/\n^i) \\ 
\!\!&\cong&\!\! M/E_i \cong \bigoplus_{0 \le j < i}M_j. 
\end{eqnarray*}
Set $F_i = E_i \otimes_A R$, and identify $F_i$ with its image in $K=M\otimes_A R$. Then $F_i\cong\widehat{E_i}$, and, under the
identification $K=M \otimes_A R\cong \prod_{j \ge 0}M_j$, we have
$$
F_i = \{(g_j)_{j \ge 0} \in  \prod_{j \ge 0}M_j \mid g_0 = \cdots = g_{i-1}=0\} \cong \prod_{j \ge i}M_j. 
$$
Since $E_i = \sum_{j=1}^r\n^{q_j}x_j$, we have
$F_i = \sum_{j=1}^r \m^{q_j}(x_j \otimes 1)$. 
Thus, $\m F_i\subseteq F_{i+1}$ for every $i\ge 0$, with equality for $i\geq b$. 
Moreover, extending the inclusions $\n^i M \subseteq E_i \subseteq \n^{\max\{i-b, 0\}}M$ to $R$, we obtain $\m^i K \subseteq F_i \subseteq \m^{\max\{i-b, 0\}}K$ for all $i \ge 0$. 
Identifying $F_{i+1}$ with its image in $F_i$, we see that $\calF=\{F_i\}_{i\geq0}$ is a good $\m$-filtration of $K$. 
Furthermore, under the natural identification $\gr_{\m}(R)\cong A$, the above description of $F_i$ yields an isomorphism
$$
\gr_{\calF}(K) = \bigoplus_{i \ge 0}F_i/F_{i+1} \cong \bigoplus_{i \ge 0}M_i = M
$$
of graded $A$-modules.

\medskip

The following plays a key role in the proof of Theorem \ref{main}. The first two assertions are essentially due to \cite[Section 2]{Higashitani}. Although the base field is assumed there to be algebraically closed of characteristic zero, the proofs of the relevant results do not use this assumption. We therefore include brief proofs for completeness.

\begin{lem}\label{key}
Assume that $A$ is not Gorenstein. Then the following assertions hold true. 
\begin{itemize}
\item[$(1)$] The multiplicity $\rme(M/Au)$ is independent of the choice of $u\in M_0\setminus\{0\}$. We denote this common value by $c$.
\item[$(2)$] One has $c \ge r-1$, and the graded ring $A$ is almost Gorenstein if and only if $c = r-1$. 
\item[$(3)$] For every $f\in K \setminus \{0\}$, the equality $\rme^0_{\m}(K/Rf)
=c+\ord_{\calF}(f)\cdot\rme(A)$ holds, where $\ord_{\calF}(f) = \max\{j \ge 0 \mid f \in F_j\}$. 
\end{itemize}
\end{lem}

\begin{proof}
Since $A$ is not Gorenstein, we have $d\ge 1$.

$(1)$ Let $u\in M_0\setminus\{0\}$. Since $A$ is an integral domain and $M=\rmK_A(-a)$ is torsion-free, the homomorphism $\varphi: A \to M$ defined by $\varphi(1)=u$ is injective. Hence the graded exact sequence $0 \to A \overset{\varphi}{\to} M \to M/Au \to 0$ of $A$-modules yields
$$
H_{M/Au}(t)= H_M(t) - H_A(t). 
$$
Since the right-hand side is independent of the choice of $u\in M_0\setminus\{0\}$, so is the Hilbert series of $M/Au$. In particular, the multiplicity $\rme(M/Au)$ is independent of the choice of $u$. We denote this common value by $c$.

$(2)$ Let $u\in M_0\setminus\{0\}$, and set $D=M/Au$. Since $A$ is not Gorenstein, we have $D\ne(0)$. Then $D$ is a Cohen-Macaulay $A$-module with $\dim_A D=d-1$ (see \cite[Lemma 3.1]{GTT}). Moreover, since $M_i=(0)$ for all $i<0$, we have $(A_+M)_0=(0)$. Hence $u\notin A_+M=\n M$, and therefore $\mu_A(D) = r-1$. By $(1)$, we have $c=\rme(D)$. Thus
$$
c = \rme(D) = \rme^0_{\n A_\n}(D_\n) \ge \mu_{A_\n}(D_{\n}) = \mu_A(D) =r-1. 
$$
By the above observation, $A$ is almost Gorenstein if and only if there exists $u\in M_0\setminus\{0\}$ such that $\rme(M/Au)=r-1$. By $(1)$, the latter condition is equivalent to $c=r-1$.

$(3)$ Set $C=K/Rf$. For each $i\ge 0$, define $\overline{F}_i = (F_i + Rf)/Rf$ and $G_i = Rf \cap F_i$. Then $\overline{\calF}=\{\overline{F}_i\}_{i\ge 0}$ and
$\calG=\{G_i\}_{i\ge 0}$ are $\m$-filtrations of $C$ and $Rf$, respectively (see \cite[Definition 8.1]{GOTWY}). Moreover, the inclusion $\iota: Rf \to K$ and the natural surjection $\varepsilon:K\to C$ are strict homomorphisms. By \cite[Lemma 8.4]{GOTWY}, the exact sequence $0 \to Rf \to K \to C \to 0$ induces an exact sequence
$$
0 \to \gr_{\calG}(Rf) \to \gr_{\calF}(K) \to \gr_{\overline{\calF}}(C) \to 0
$$
of graded modules.
We now determine the image of the composite map
$$
\xi:\gr_{\calG}(Rf)\to\gr_{\calF}(K)\cong M.
$$
Set $q=\ord_{\calF}(f)$. Then $f\in F_q\setminus F_{q+1}$. Let $f^*\in M_q$ denote the image of $f+F_{q+1}\in F_q/F_{q+1}$
via the natural isomorphism $F_q/F_{q+1}\cong M_q$. Recall that, under the identification $K \cong \prod_{j \ge 0}M_j$, we have $F_i =  \{(g_j)_{j \ge 0} \in  K \mid g_0 = \cdots = g_{i-1}=0\}$. Thus, writing $f=(f_j)_{j\geq0}$ with $f_j\in M_j$, it follows that $q = \min \{j \ge 0 \mid f_j \ne 0\}$ and $f^* = f_q$.

For every $i<q$, since $Rf\subseteq F_q\subseteq F_{i+1}$, we have $(\Im \xi)_i=(0)$. Let $i\geq q$, and take $\alpha\in(\Im \xi)_i$. Write $\alpha = \xi(\overline{rf})$ for some $r\in R$ with $rf\in F_i$, where $\overline{rf}$ denotes 
the image of $rf$ in $G_i/G_{i+1}$. We may assume that $\alpha\ne0$. Then $rf\in F_i\setminus F_{i+1}$ and $r\ne0$, so that $i=\ord_{\calF}(rf)$. 
Under the identification $R\cong\prod_{j\geq0}A_j$, write $r=(r_j)_{j\geq0}$ with $r_j\in A_j$, and set $p=\min\{j\geq0\mid r_j\ne0\}$. Since
$$
f = (0, \ldots, 0, f_q, f_{q+1}, \ldots) \ \ \text{with} \ \  f_q = f^* \ne 0,
$$
the homogeneous component of $rf$ of degree $p+q$ is $r_pf_q=r_pf^*$. Since $A$ is an integral domain and $M$ is torsion-free, we have $r_pf^*\ne 0$. It follows that $\ord_{\calF}(rf) = p + q$ and $(rf)^* = r_p f^*$. Hence $p=i-q$, and therefore
$\alpha=r_{i-q}f^*\in A_{i-q}f^*$. Thus $(\Im \xi)_i\subseteq A_{i-q}f^*$. Conversely, if $0 \ne a\in A_{i-q}$, then $\ord_{\calF}(af)=i$ and $(af)^*=af^*$. Hence $(\Im \xi)_i=A_{i-q}f^*$ for every $i\geq q$. 
Consequently, $\Im \xi = \bigoplus_{i \ge q}A_{i-q}f^* = Af^*$. Since $\gr_{\calF}(K)\cong M$, the above exact sequence now yields an isomorphism
$$
\gr_{\overline{\calF}}(C) \cong M/Af^*
$$
of graded $A$-modules. 
Note that $\overline{\calF}$ is a good $\m$-filtration of $C$. Indeed
$$
\overline{F}_{i+1}
=\frac{F_{i+1}+Rf}{Rf}
=\frac{\m F_i+Rf}{Rf}
=\m\overline{F}_i
$$
for every $i\geq b$.
Hence $\rme_{\overline{\calF}}(C) = \rme^0_{\m}(C)$, so that $\rme^0_{\m}(C) = \rme_{\overline{\calF}}(C) = \rme(\gr_{\overline{\calF}}(C)) = \rme(M/Af^*)$. 

Finally, since $f^*\in M_q\setminus\{0\}$, we have an exact sequence 
$$
0 \to A(-q) \to M \to M/Af^* \to 0
$$
of graded $A$-modules. Choose $0\ne u\in M_0$, and set $D=M/Au$, so that $\rme(D)=c$. 
Thus, $H_{M/Af^*}(t)=H_M(t)-t^qH_A(t)=H_D(t)+(1-t^q)H_A(t)$. 
Write 
$$
H_D(t)=\frac{h_D(t)}{(1-t)^{d-1}} \ \ \text{with} \ \ h_D(t) \in \Bbb Z[t]. 
$$ 
If $q\ge 1$, we then have 
$$
H_{M/Af^*}(t)=\frac{h_D(t)}{(1-t)^{d-1}}+
\frac{(1+t+\cdots+t^{q-1})h_A(t)}{(1-t)^{d-1}}.
$$
Therefore, $\rme(M/Af^*)=h_D(1)+q \cdot h_A(1)=\rme(D)+q\cdot\rme(A)
=c+q\cdot\rme(A)$. The same equality is immediate when $q=0$. Consequently,
$\rme^0_{\m}(C) =c + \ord_{\calF}(f)  \cdot\rme(A)$, as desired.
\end{proof}


We are now ready to prove Theorem \ref{main}. 

\begin{proof}[Proof of Theorem \ref{main}]
We retain the notation of Setup \ref{setup}, together with that introduced thereafter. We may assume that $A$ is not Gorenstein. Thus, $d=\dim A\ge 1$. Suppose first that $A$ is almost Gorenstein as a graded ring. Then the local ring $A_{\n}$ is  almost Gorenstein. Hence, by \cite[Theorem 3.8]{GTT}, its completion $R=\widehat{A_{\n}}$ is also almost Gorenstein. 
Conversely, suppose that $R$ is almost Gorenstein. Then there exists an exact sequence
$$
0 \to R \overset{\varphi}{\to} K \to C \to 0
$$
of $R$-modules such that $\rme^0_{\m}(C) = \mu_R(C)$. Set $f=\varphi(1)$. Then $f\in K\setminus\{0\}$ and $C\cong K/Rf$. Hence, by Lemma \ref{key} (3), it follows that
$$
\rme^0_{\m}(C)=\rme^0_{\m}(K/Rf)
=c+q\cdot\rme(A)
$$
where $q=\ord_{\calF}(f)$.
We claim that $q=0$. Suppose, to the contrary, that $q\geq1$. Since $\mu_R(C)=\mu_R(K/Rf) \le \mu_R(K) =r$ and $\rme(A)\ge 2$, Lemma \ref{key} (2) yields
$$
\rme^0_{\m}(C) \ge (r-1) + q\cdot\rme(A) \ge (r-1) + \rme(A) \ge r+1 > r \ge \mu_R(C)
$$
which contradicts $\rme_{\m}^0(C)=\mu_R(C)$. Thus $q=0$, so that
$f\in F_0 \setminus F_1$. As $\m K\subseteq F_1$, we get $f\notin\m K$, and hence $\mu_R(C)=\mu_R(K/Rf)=r-1$. 
Consequently
$$
c = \rme^0_{\m}(C) = \mu_R(C) = r-1. 
$$
By Lemma \ref{key} (2), we conclude that $A$ is almost Gorenstein as a graded ring.
\end{proof}

\begin{rem}\label{2.3}
Suppose that the field $k$ is infinite. By \cite[Theorem 3.9]{GTT}, the almost Gorenstein property of $A_{\n}$ is equivalent to that of its completion $R$. Combining this with Theorem \ref{main}, it follows that $A$ is almost Gorenstein as a graded ring if and only if $A_{\n}$ is almost Gorenstein as a local ring.
\end{rem}

The preceding argument yields the following quantitative refinement of Theorem \ref{main}.

\begin{cor}
Assume that $A$ is not Gorenstein. Then
$$
\min_{0\ne f\in K}\left\{\rme_{\m}^0(K/Rf)-\mu_R(K/Rf)\right\}=c-r+1.
$$
Moreover, for $0\ne f\in K$, the minimum is attained if and only if
$\ord_{\calF}(f)=0$.
\end{cor}

\begin{proof}
The assertion follows from Lemma \ref{key}~(3) and the argument in the
proof of Theorem \ref{main}. Indeed, if $\ord_{\calF}(f)=0$, then $\mu_R(K/Rf)=r-1$, whereas if $\ord_{\calF}(f)\geq1$, then $\rme_{\m}^0(K/Rf)-\mu_R(K/Rf)\geq c-r+2$. 
\end{proof}

\section{Corollaries and Examples}

Recall that a Cohen-Macaulay local ring $(R, \m)$ has {\it minimal multiplicity} if $\rme(R) = \mu_R(\m) - \dim R + 1$. 

\begin{cor}
Let $A=\bigoplus_{i \ge 0}A_i$ be a Cohen-Macaulay standard graded integral domain over an infinite field $k=A_0$ and $\n$ the graded maximal ideal of $A$. Assume that $A$ is level but not Gorenstein. Then the $\n A_{\n}$-adic completion $\widehat{A_{\n}}$ is almost Gorenstein as a local ring if and only if it has minimal multiplicity. 
\end{cor}

\begin{proof}
The assertion follows from \cite[Theorem 1.6]{GTT} and Theorem \ref{main}. 
\end{proof}

Theorem \ref{main}, together with Remark \ref{2.3} and \cite[Theorem 1.6]{GTT}, recovers the following results on generic and symmetric determinantal rings. 

\begin{cor}[{\cite[Theorem 1.1]{Taniguchi}}]
Let $m$, $n \ge 2$ be integers and $X = \left[X_{ij}\right]$ be an $m\times n$ matrix of indeterminates over an infinite field $k$. We denote by $S=k[X]$ the polynomial ring generated by $\{X_{ij}\}_{1 \le i \le m,\,1\le j \le n}$ over $k$. Let $\rmI_t(X)$ be the ideal of $S$ generated by the $t \times t$-minors of $X$, where $2 \le t \le \min\{m, n\}$. Set $A = S/{\rm I}_t(X)$ and $\n = A_+$.  
Then the following conditions are equivalent.
\begin{itemize}
\item[$(1)$] $A$ is an almost Gorenstein graded ring.
\item[$(2)$] $A_{\n}$ is an almost Gorenstein local ring.
\item[$(3)$] Either $m=n$, or $m\ne n$ and $t=\min\{m, n\}=2$.
\end{itemize}
\end{cor}

\begin{cor}[{\cite[Theorem 1.1]{CELW}}]
Let $n > 0$ be an integer and $X = \left[X_{ij}\right]$ be an $n \times n$ symmetric matrix of indeterminates over an infinite field $k$. We denote by $S = k[X]$ the polynomial ring over $k$ generated by $\{X_{ij}\}_{1 \le i,\, j \le n}$.  
Let $\rmI_{t+1}(X)$ be the ideal of $S$ generated by $(t + 1) \times (t + 1)$ minors of $X$, where $0 < t < n$. Set $A = S/\rmI_{t+1}(X)$ and $\n = A_+$. 
Then the following conditions are equivalent.
\begin{itemize}
\item[$(1)$] $A$ is an almost Gorenstein graded ring.
\item[$(2)$] $A_{\n}$ is an almost Gorenstein local ring.
\item[$(3)$] Either $n - t$ is odd, or $n = 3$ and $t = 1$. 
\end{itemize}
\end{cor}

We next illustrate Theorem \ref{main} with an example of a level ring.

\begin{ex}[{\cite[Example 10.8]{GTT}}]
Let $A= k[X_1, X_2, \ldots , X_d]~ (d \ge 1)$ be the polynomial ring over an infinite field $k$. Let $\ell \ge 1$ be an integer, and set $A^{(\ell)} = k[A_\ell]$ the Veronese  subring of $A$. Let $R$ be the $\n {A^{(\ell)}}_\n$-adic completion of the local ring ${A^{(\ell)}}_\n$, where $\n$ denotes the graded maximal ideal of $A^{(\ell)}$. Then the following assertions hold true. 
\begin{itemize}
\item[$(1)$] The local ring $R$ is almost Gorenstein if $d \le 2$.  
\item[$(2)$] Suppose that $d \ge 3$. Then $R$ is an almost Gorenstein local ring if and only if either $\ell \mid d$, or $d=3$ and $\ell=2$. 
\end{itemize}
In particular, if $d\geq 4$ and $\ell\nmid d$, then $R$ is nearly Gorenstein but not almost Gorenstein (see \cite[Lemma 1.5, Corollary 4.8]{HHS}).
\end{ex}

We conclude this paper with an example of a non-level ring to which Theorem \ref{main} applies. 

\begin{ex}
Let $S=k[X_1, X_2, \ldots,X_n]~(n\ge 5)$ be the polynomial ring over a field $k$.  Set
$$
B_n=k[X_iX_j\mid 1\leq i<j\leq n]\subseteq S \ \ \text{with} \ \ \deg(X_iX_j)=1. 
$$
Then $B_n$ is a Cohen-Macaulay standard graded normal domain. Let
$
T_n=\widehat{(B_n)_{(B_n)_+}}
$
be the $(B_n)_+(B_n)_{(B_n)_+}$-adic completion of the local ring 
$(B_n)_{(B_n)_+}$. By \cite[Theorem 1.3]{HM}, together with
Theorem \ref{main}, the local ring $T_n$ is almost Gorenstein if and only if $n$ is even. 
When this is the case, $B_n$ is not level (see \cite[Theorem 1.2]{HM}). 
\end{ex}



\ac{The author would like to express his sincere gratitude to Ryo Takahashi and Ken-ichi Yoshida for independently raising the question that motivated this work, as well as for their valuable comments on the manuscript.}


\end{document}